\documentclass[11pt,letterpaper]{article}

\usepackage{fullpage}

\usepackage{graphicx,subfigure}

\def\showauthornotes{2}

\def\showkeys{0}
\def\showdraftbox{0}
\def\showcolorlinks{1}
\def\usemicrotype{1}
\def\showfixme{0}

\def\arxivmode{0}
\def\fastmode{0}

\newcommand{\llangle}{\left\langle}
\newcommand{\rrangle}{\right\rangle}
\newcommand{\sC}{\mathscr{C}}

\ifnum\fastmode=0
\usepackage{etex}
\fi

\ifnum\fastmode=0
\usepackage[l2tabu, orthodox]{nag}
\fi

\usepackage{xspace,enumerate}

\usepackage[dvipsnames]{xcolor}

\ifnum\fastmode=0
\usepackage[T1]{fontenc}
\usepackage[full]{textcomp}
\fi

\usepackage[american]{babel}

\usepackage{mathtools}

\usepackage{amsthm}

\newtheorem{theorem}{Theorem}[section]
\newtheorem*{theorem*}{Theorem}

\newtheorem*{proposition*}{Proposition}
\newtheorem{lemma}[theorem]{Lemma}
\newtheorem*{lemma*}{Lemma}
\newtheorem{corollary}[theorem]{Corollary}
\newtheorem*{conjecture*}{Conjecture}
\newtheorem{fact}[theorem]{Fact}
\newtheorem*{fact*}{Fact}

\newtheorem*{exercise*}{Exercise}

\newtheorem*{hypothesis*}{Hypothesis}

\theoremstyle{definition}
\newtheorem{definition}[theorem]{Definition}

\newtheorem{example}[theorem]{Example}

\newtheorem{assumption}[theorem]{Assumption}
\newtheorem{exercise-easy}[theorem]{Exercise}
\newtheorem{exercise-med}[theorem]{Exercise}
\newtheorem{exercise-hard}[theorem]{Exercise$^\star$}

\newtheorem*{claim*}{Claim}

\newtheorem{remark}[theorem]{Remark}
\newtheorem*{remark*}{Remark}

\newtheorem*{observation*}{Observation}

\ifnum\arxivmode=0
\usepackage{newpxtext} %
\usepackage{textcomp} %
\usepackage[varg,bigdelims]{newpxmath}
\usepackage[scr=rsfso]{mathalfa}%
\usepackage{bm} %
\let\mathbb\varmathbb
\fi

\ifnum\arxivmode=1
\usepackage[varg]{pxfonts} %
\usepackage{textcomp} %
\usepackage[scr=rsfso]{mathalfa}%
\usepackage{bm} %
\fi

\ifnum\showkeys=1
\usepackage[color]{showkeys}
\fi

\makeatletter
\@ifpackageloaded{hyperref}
{}
{\usepackage[pagebackref]{hyperref}}

\definecolor{bleudefrance}{rgb}{0.01, 0.1, 1.0}
\definecolor{azure}{rgb}{0.0, 0.5, 1.0}

\ifnum\showcolorlinks=1
\hypersetup{
pagebackref=true,
colorlinks=true,
urlcolor=blue,
linkcolor=blue,
citecolor=OliveGreen}
\fi

\ifnum\showcolorlinks=0
\hypersetup{
pagebackref=true,
colorlinks=false,
pdfborder={0 0 0}}
\fi

\usepackage{prettyref}

\newcommand{\savehyperref}[2]{\texorpdfstring{\hyperref[#1]{#2}}{#2}}

\newrefformat{eq}{\savehyperref{#1}{\textup{(\ref*{#1})}}}
\newrefformat{lem}{\savehyperref{#1}{Lemma~\ref*{#1}}}
\newrefformat{def}{\savehyperref{#1}{Definition~\ref*{#1}}}
\newrefformat{obs}{\savehyperref{#1}{Observation~\ref*{#1}}}
\newrefformat{ass}{\savehyperref{#1}{Assumption~\ref*{#1}}}
\newrefformat{thm}{\savehyperref{#1}{Theorem~\ref*{#1}}}
\newrefformat{cor}{\savehyperref{#1}{Corollary~\ref*{#1}}}
\newrefformat{cha}{\savehyperref{#1}{Chapter~\ref*{#1}}}
\newrefformat{sec}{\savehyperref{#1}{Section~\ref*{#1}}}
\newrefformat{app}{\savehyperref{#1}{Appendix~\ref*{#1}}}
\newrefformat{tab}{\savehyperref{#1}{Table~\ref*{#1}}}
\newrefformat{fig}{\savehyperref{#1}{Figure~\ref*{#1}}}
\newrefformat{hyp}{\savehyperref{#1}{Hypothesis~\ref*{#1}}}
\newrefformat{alg}{\savehyperref{#1}{Algorithm~\ref*{#1}}}
\newrefformat{rem}{\savehyperref{#1}{Remark~\ref*{#1}}}
\newrefformat{item}{\savehyperref{#1}{\ref*{#1}}}
\newrefformat{step}{\savehyperref{#1}{step~\ref*{#1}}}
\newrefformat{conj}{\savehyperref{#1}{Conjecture~\ref*{#1}}}
\newrefformat{fact}{\savehyperref{#1}{Fact~\ref*{#1}}}
\newrefformat{prop}{\savehyperref{#1}{Proposition~\ref*{#1}}}
\newrefformat{prob}{\savehyperref{#1}{Problem~\ref*{#1}}}
\newrefformat{claim}{\savehyperref{#1}{Claim~\ref*{#1}}}
\newrefformat{relax}{\savehyperref{#1}{Relaxation~\ref*{#1}}}
\newrefformat{red}{\savehyperref{#1}{Reduction~\ref*{#1}}}
\newrefformat{part}{\savehyperref{#1}{Part~\ref*{#1}}}
\newrefformat{ex}{\savehyperref{#1}{Exercise~\ref*{#1}}}
\newrefformat{property}{\savehyperref{#1}{Property~\ref*{#1}}}

\newcommand{\Sref}[1]{\hyperref[#1]{\S\ref*{#1}}}

\usepackage{nicefrac}

\ifnum\fastmode=0
\ifnum\usemicrotype=1
\usepackage{microtype}
\fi
\fi

\ifnum\showauthornotes=2
\newcommand{\mynotes}[1]{{\sffamily\small\color{teal}{#1}}\medskip}
\newcommand{\Authornote}[2]{{\sffamily\small\color{Maroon}{[#1: #2]}}\medskip}
\newcommand{\Authornotecolored}[3]{{\sffamily\small\color{#1}{[#2: #3]}}}
\newcommand{\Authorcomment}[2]{{\sffamily\small\color{gray}{[#1: #2]}}}
\newcommand{\Authorstartcomment}[1]{\sffamily\small\color{gray}[#1: }

\newcommand{\Authorfnote}[2]{\footnote{\color{red}{#1: #2}}}
\newcommand{\Authorfixme}[1]{\Authornote{#1}{\textbf{??}}}
\newcommand{\Authormarginmark}[1]{\marginpar{\textcolor{red}{\fbox{\Large #1:!}}}}
\newcommand{\myexplain}[1]{{\sffamily\small\color{red}{\noindent [Explanation:\medskip\newline \begin{quote}#1\hfill]\end{quote}}}\medskip}
\newcommand{\explain}[1]{{\sffamily\small\color{red}{#1}}\medskip}

\else

\newcommand{\mynotes}[1]{}
\newcommand{\Authornote}[2]{}
\newcommand{\Authornotecolored}[3]{}
\newcommand{\Authorcomment}[2]{}
\newcommand{\Authorstartcomment}[1]{}

\newcommand{\Authorfnote}[2]{}
\newcommand{\Authorfixme}[1]{}
\newcommand{\Authormarginmark}[1]{}
\newcommand{\myexplain}[1]{}
\newcommand{\explain}[1]{}

\ifnum\showauthornotes=1
\renewcommand{\myexplain}[1]{{\sffamily\small\color{red}{\noindent \begin{quote}{\bf Explanation:} \medskip\newline #1\end{quote}}}\medskip}
\fi

\fi

\ifnum\showfixme=0

\fi

\usepackage{boxedminipage}

\newcommand{\Esymb}{\mathbb{E}}
\newcommand{\Psymb}{\mathbb{P}}

\DeclareMathOperator*{\E}{\Esymb}

\DeclareMathOperator*{\ProbOp}{\Psymb}

\renewcommand{\Pr}{\ProbOp}

\newcommand{\textparen}[1]{\text{(#1)}}

\ifx\because\undefined
\newcommand{\because}[1]{\textparen{because #1}}
\else
\renewcommand{\because}[1]{\textparen{because #1}}
\fi

\newcommand{\seteq}{\mathrel{\mathop:}=}

\newcommand{\bigmid}{~\big|~}

\newcommand\bdot\bullet

\ifx\mathds\undefined %

\else

\fi

\DeclareMathOperator{\dist}{dist}

\newcommand{\Z}{\mathbb Z}
\newcommand{\N}{\mathbb N}
\newcommand{\R}{\mathbb R}

\newcommand{\cC}{\mathcal C}

\newcommand{\cE}{\mathcal E}

\newcommand{\cG}{\mathcal G}

\newcommand{\cL}{\mathcal L}
\newcommand{\cM}{\mathcal M}

\newcommand{\cS}{\mathcal S}

\newcommand{\sfB}{\mathsf B}

\newcommand{\bbZ}{\mathbb Z}

\renewcommand{\leq}{\leqslant}
\renewcommand{\le}{\leqslant}
\renewcommand{\geq}{\geqslant}
\renewcommand{\ge}{\geqslant}

\ifnum\showdraftbox=1

\else

\fi

\let\epsilon=\varepsilon

\numberwithin{equation}{section}

\newcommand\MYcurrentlabel{xxx}

\newcommand{\MYstore}[2]{%
  \global\expandafter \def \csname MYMEMORY #1 \endcsname{#2}%
}

\newcommand{\MYload}[1]{%
  \csname MYMEMORY #1 \endcsname%
}

\newcommand{\MYnewlabel}[1]{%
  \renewcommand\MYcurrentlabel{#1}%
  \MYoldlabel{#1}%
}

\newcommand{\MYdummylabel}[1]{}

\newcommand{\torestate}[1]{%
  \let\MYoldlabel\label%
  \let\label\MYnewlabel%
  #1%
  \MYstore{\MYcurrentlabel}{#1}%
  \let\label\MYoldlabel%
}

\newcommand{\restatetheorem}[1]{%
  \let\MYoldlabel\label
  \let\label\MYdummylabel
  \begin{theorem*}[Restatement of \prettyref{#1}]
    \MYload{#1}
  \end{theorem*}
  \let\label\MYoldlabel
}

\newcommand{\restatelemma}[1]{%
  \let\MYoldlabel\label
  \let\label\MYdummylabel
  \begin{lemma*}[Restatement of \prettyref{#1}]
    \MYload{#1}
  \end{lemma*}
  \let\label\MYoldlabel
}

\newcommand{\restateprop}[1]{%
  \let\MYoldlabel\label
  \let\label\MYdummylabel
  \begin{proposition*}[Restatement of \prettyref{#1}]
    \MYload{#1}
  \end{proposition*}
  \let\label\MYoldlabel
}

\newcommand{\restatefact}[1]{%
  \let\MYoldlabel\label
  \let\label\MYdummylabel
  \begin{fact*}[Restatement of \prettyref{#1}]
    \MYload{#1}
  \end{fact*}
  \let\label\MYoldlabel
}

\newcommand{\restate}[1]{%
  \let\MYoldlabel\label
  \let\label\MYdummylabel
  \MYload{#1}
  \let\label\MYoldlabel
}

\newcommand{\addreferencesection}{
  \phantomsection
\ifnum\stocmode=0
  \addcontentsline{toc}{section}{References}
\else
  \addcontentsline{toc}{section}{References \hspace*{1in} --------- End of extended abstract ---------}
\fi

}

\newcommand{\e}{\epsilon}
\newcommand{\eps}{\epsilon}

\renewcommand{\paragraph}[1]{\medskip\noindent{\bf #1.}}

\allowdisplaybreaks

\usepackage{paralist}

\usepackage{comment}

\let\pref=\prettyref

\newcommand{\diam}{\mathrm{diam}}

\newcommand{\vertiii}[1]{{\left\vert\kern-0.25ex\left\vert\kern-0.25ex\left\vert #1 
          \right\vert\kern-0.25ex\right\vert\kern-0.25ex\right\vert}}

\newcommand\f{\varphi}

\usepackage{enumitem}
\usepackage[normalem]{ulem}
\usepackage{stmaryrd}

\newcommand{\cmnt}[1]{}

\newcommand{\len}{\mathrm{len}}

\newcommand{\1}{\bm{1}}

\newcommand{\conn}[1]{\overset{#1}{\leftrightarrow}}
\newcommand{\mult}{\mathsf{m}}
\newcommand{\base}{\mathscr{G}}
\renewcommand{\mathbb}{\vvmathbb}

\newcommand{\restrict}{\!\!\upharpoonright}

\title{Chemical subdiffusivity of critical 2D percolation}
\author{Shirshendu Ganguly \thanks{University of California, Berkeley, \texttt{sganguly@berkeley.edu}} \and James R. Lee\thanks{University of Washington, \texttt{jrl@cs.washington.edu}}}
\date{}

\begin{document}

\maketitle

\begin{abstract}
   We show that random walk on the incipient infinite cluster (IIC) of two-dimensional 
   critical percolation is subdiffusive in the chemical distance
   (i.e., in the intrinsic graph metric).
   Kesten (1986) famously showed that this is true for the Euclidean distance, but it is known
   that the chemical distance is typically asymptotically larger.

   More generally, we show that subdiffusivity in the chemical distance
   holds for stationary random graphs of polynomial volume growth, as long
   as there is a multiscale way of covering the graph so that ``deep patches''
   have ``thin backbones.''
   Our estimates are quantitative and give explicit bounds in terms
   of the one and two-arm exponents $\eta_2 > \eta_1 > 0$:
   For $d$-dimensional models, the mean chemical displacement after
   $T$ steps of random walk scales asymptotically slower than $T^{1/\beta}$,
   whenever
   \[
      \beta < 2 + \frac{\eta_2-\eta_1}{d-\eta_1}\,.
   \]
   Using the conjectured values of $\eta_2 = \eta_1 + 1/4$ and $\eta_1 = 5/48$ for
   2D lattices, the latter quantity is $2+12/91$.
\end{abstract}

\begingroup
\hypersetup{linktocpage=false}
\setcounter{tocdepth}{2}
\tableofcontents
\endgroup

\newpage

\section{Introduction}
\label{sec:intro}

Bond percolation on a graph $G$ with parameter $p \in [0,1]$
is a probability measure on random subgraphs of $G$
obtained by retaining each edge independently with probability $p$.
Denoting by $\bm{G}_p$ such a random subgraph,
the Kolmogorov $0$-$1$ law implies that
\[
   \Pr[\bm{G}_p \textrm{ contains an infinite connected component}] \in \{0,1\}.
\]
Thus by monotonicity with respect to $p$, there is a {\em critical} value $p_c(G) \in [0,1]$
such that for $p < p_c(G)$, all components of $\bm{G}_p$ are finite almost surely,
and for $p > p_c(G)$, there is almost surely an infinite cluster.

Consider now the case $G=\bbZ^d$.  It has been established that for $d=2$ \cite{Kesten86iic}
and for $d$ sufficiently large \cite{HS90}, almost surely critical percolation does not
produce an infinite cluster.  Nevertheless, it is known that at criticality
there are arbitrarily large clusters \cite{Aizenman97}, and these
large percolation clusters are conjectured to exhibit fractal characteristics.

One such characteristic is that of {\em anomalous diffusion,} meaning that the random walk
spreads out slower than a corresponding walk in Euclidean space.
In $\bbZ^d$, the expected distance of the random walk from its starting point at time $t$ scales
like $t^{1/2}$, and one refers to the walk on a graph $G$ as {\em subdiffusive} if the typical
distance scales asymptotically slower.
A convenient method to study the geometry of large percolation clusters 
is to take a suitable weak limit of
typical clusters with sizes tending to $\infty$ (see \cite{Kesten86iic} for $d=2$ and \cite{HJ04} for $d$ sufficiently large).
The resulting infinite random graph is known as the {\em incipient infinite cluster} (IIC).

Let $G$ denote a connected graph, $\dist_G$ the path metric on $G$, and $\{X_t\}$ the standard random walk on $G$.
For $d$ sufficiently large, the solution of the Alexander-Orbach conjecture in high dimensions \cite{kn09}
shows that when $G$ has the law of the IIC in $\Z^d$,
\[
   \E[\dist_{G}(X_0,X_t)^2] \asymp t^{2/3}.
\]

The first mathematical results in this area are due to Kesten \cite{Kesten86}
who showed much earlier
that 
when $G$ has the law of the IIC in $\Z^2$, there
is some $\e > 0$ for which
\begin{equation}\label{eq:kesten-speed}
   \E\left[\|X_0-X_t\|^2_2\right] \leq C t^{1-\eps}, \quad t \geq 1.
\end{equation}

\paragraph{Chemical vs. Euclidean displacement for the IIC in $\Z^2$}
Note that Kesten's bound on the speed is in the {\em Euclidean distance,}
as opposed to the graph metric $\dist_G$.
One observes that \eqref{eq:kesten-speed} does not necessarily yield a subdiffusive
bound on the speed in $\dist_G$.  Indeed, it is known from work of Aizenman and Burchard
\cite{AB99} that
for some $s > 0$, and any $u,v \in \Z^2$ with $\|u-v\|_1 = h$,
\[
   \Pr\left[\dist_G(u,v) \geq h^{1+s} \mid u,v \in V(G)\right] \to 1 \textrm{ as } h \to \infty.
\]
The strongest previous bound of which we are aware holds
for general planar, invariantly amenable, stationary random graphs $(G,X_0)$
(cf. \cite{dlp13} and \cite{lee17a}):
\[
   \E[\dist_G(X_0,X_t)^2] \leq O(t).
\]
(Indeed, one can simply apply \pref{thm:mt} below with $\omega \equiv 1$.)
We refer to \pref{sec:thin-backbone} for a review of stationary random graphs.

Asymptotic scaling of the mean-squared chemical displacement
$\E[\dist_G(X_0,X_t)^2] \asymp t^{2/\beta}$ on the IIC has been studied
experimentally
via simulations and enumeration of the random walk for small times (one can consult
the discussion and references in \cite[\S 3.7.2]{BH91}).

There are further reasons to consider specifically the chemical displacement.
Since critical percolation on the 2D lattice is conjectured to 
be ``conformally invariant,'' distances in the
standard Euclidean embedding might not be considered ``canonical.''
Moreover, speed of the random walk in the intrinsic graph metric
has been used to study the space of harmonic functions on the graph
\cite{benjamini}.

\paragraph{Thin backbones and arm exponents}
Consider the IIC as a random subgraph $G$ of $\Z^2$ containing the origin $\bm{0}$.
Define $\cS(n) \seteq [-n,n]^2$ and let $\partial \cS(n)$ denote the vertex boundary of $\cS(n)$.
Write $v \leftrightarrow S$ if there is a path in $G$ from $v \in \Z^2$ to $S \subseteq \Z^2$.
Define the quantities
\begin{align*}
   \pi_1(n) &\seteq \Pr\left(\bm{0} \leftrightarrow \partial \cS(n)\right), \\
   \pi_2(n) &\seteq \Pr\left(\bm{0} \textrm{ is connected to $\partial \cS(n)$ via two disjoint paths}\right).
\end{align*}
These are known as the one-arm and two-arm (monochromatic) arm probabilities in critical percolation.
If we use $\pi^{\mathrm{hex}}_i(n)$ to denote the analogous values on the hexagonal lattice,
it is known that $\pi^{\mathrm{hex}}_1(n) = n^{-5/48+o(1)}$ \cite{LSW02} and it is conjectured that
$\pi^{\mathrm{hex}}_2(n)$ is $n^{-5/48-1/4+o(1)}$ \cite{BN11}.
These exponents are conjectured to be universal, but much less is known for the square lattice.
Somewhat weaker bounds were derived by Kesten.

\begin{theorem}[{\cite[Lem 3.20]{Kesten86}}]
   \label{thm:kesten-arm}
   There are constants $C > c > 0$ and $\eta_1 > 0$ such that for all $n \geq 1$,
   \begin{align*}
      c n^{-1/2} &\leq \pi_1(n) \leq C n^{-\eta_1}, \\
      \frac{1}{2n} &\leq \pi_2(n) \leq 16 \pi_1(n)^2 \leq C \pi_1(n) n^{-\eta_1}.
   \end{align*}
\end{theorem}

Kesten used these bounds to prove that the random walk on the IIC is subdiffusive
in the Euclidean distance.  If we define the backbone at scale $n$ by
\[
   \sfB(n) \seteq \left\{x \in \cS(n) \cap V(G) : \textrm{there are two disjoint paths from $x$ to $\partial \cS(n)$}\right\},
\]
then \pref{thm:kesten-arm} can be used to show (cf. \pref{lem:kesten})
that $\sfB(n)$ typically occupies an asymptotically small fraction of $\cS(n) \cap V(G)$.
The random walk can only make progress in moving away from the origin by walking on the ``thin'' backbone $\sfB(n)$.
It turns out that the random walk spends the majority of its time walking in the ``bushes'' that hang off the backbone,
i.e., the vertices $(\cS(n) \cap V(G)) \setminus \sfB(n)$.

\paragraph{Conformal weights and Markov type}
Presently, we use an approach of the second named author \cite{lee17a} based on the theory of
Markov type and ``discrete conformal weights.''
Denote $V(n) \seteq \cS(n) \cap V(G)$.
Roughly speaking, one defines a weight $\omega : V(n) \to \R_+$
such that $\frac{1}{|V(n)|} \sum_{x \in V(n)} \omega(x)^2 \leq 1$.
In this case, the theory of Markov type \cite{Ball92} for weighted planar graphs \cite{dlp13} shows that
\begin{equation}\label{eq:markovtype1}
   \E[\dist_{\omega}(Y_0,Y_t)^2] \leq O(t) \E[\dist_{\omega}(Y_0,Y_1)^2]\ \leq O(t),
\end{equation}
where $\{Y_t\}$ is the random walk restricted to $V(n)$, $Y_0$ has the law of the stationary measure,
and $\dist_{\omega}$ is the shortest-path distance in $V(n)$,
where the length of a path $\gamma$ is $\sum_{x \in \gamma} \omega(x)$.
In other words, the random walk is at most diffusive in the ``conformal metric'' $\dist_{\omega}.$

Suppose one can find such a weight $\omega$ such that $\dist_\omega(x,y) \geq \dist_{G}(x,y)^{1+\epsilon}$ for some $\epsilon > 0$ and all $x,y \in V(n)$.
This yields the bound
\[
   \E[\dist_G(Y_0,Y_t)^2] \leq t^{1/(1+\epsilon)},
\]
suggesting that the random walk is subdiffusive.
We show how to achieve this
by choosing $\omega$ to be a mixture of weights---one per scale---so
that each weight is supported on an appropriately defined backbone at that scale. 

For the reader encountering Markov type for the first time, let us give a bit more intuition.
Note that since $V(n)$ has uniformly bounded vertex degrees, it holds that
\[
   \frac{1}{|V(n)|} \sum_{x \in V(n)} \omega(x)^2 \leq O(1) \iff \E[\dist_{\omega}(Y_0,Y_1)^2] \leq O(1),
\]
yielding the second inequality in \eqref{eq:markovtype1}.

One can consider the first inequality in \eqref{eq:markovtype1} as an analog of the following fact.
Suppose that $\{Z_t\}$ is a reversible Markov chain with finite state space $\{x_1,\ldots,x_n\} \subseteq \R$ and $Z_0$
has the law of the stationary measure $\pi$.
Then for all $t \geq 1$,
\[
   \E |Z_0 - Z_t|^2 \leq t \E |Z_0-Z_1|^2.
\]
This fact can be deduced from the negative correlation inequality
\[
   \E (Z_t-Z_{t-1}) (Z_{t-1}-Z_0) \leq 0,
\]
and it is an exercise (see, e.g., \cite[\S 2]{lmn02}) to show that
\[
   \E (Z_t-Z_{t-1}) (Z_{t-1}-Z_0) = - \left\langle \bm{x}, (I-P)(I-P^{t-1}) \bm{x}\right\rangle_{\pi} \leq 0,
\]
where $\bm{x} = (x_1,\ldots,x_n)$ and $\langle u,v\rangle_{\pi} = \sum_i \pi(x_i) u_i v_i$, and $P$
is the transition matrix of the chain.
Recall that $P$ is self-adjoint with respect to the inner product $\langle \cdot,\cdot\rangle_{\pi}$,
and the latter inequality follows because $(I-P)(I-P^{t-1})$ is positive semidefinite.

Kesten's argument proceeds by showing that the random walk restricted to the backbone has at most diffusive speed in the Euclidean distance,
and that the presence of the bushes means that the walk {\em watched} on the backbone is
asymptotically slower than the walk {\em restricted} to the backbone, and is thus subdiffusive in the Euclidean distance.

If we define the weight $\omega \seteq \left(|V(n)|/|\sfB(n)|\right)^{1/2} \1_{\sfB(n)}$, then
\eqref{eq:markovtype1} asserts directly that the walk started from stationarity in $V(n)$
and watched on the backbone has average displacement $\lesssim \left(t |\sfB(n)|/|V(n)|\right)^{1/2}$ {\em in the graph distance}
after $t$ steps.
When $t \leq n^{O(1)}$ and $|V(n)|/|\sfB(n)| \gtrsim n^{\delta}$ for some $\delta > 0$,
the resulting speed is subdiffusive.

While \eqref{eq:markovtype1} holds for path metrics on planar graphs,
applying the more general theory of Markov type one only loses $(\log t)^{O(1)}$ factors and
the general result stated in \pref{thm:thick-thin} proves subdiffusivity under rather general conditions.
In particular,
this should suffice to imply subdiffusivity of the random walk on an appropriately defined IIC in higher dimensions, as soon as one can establish the existence of a
suitable backbone analogous to the planar case. A non-trivial one-arm exponent seems to be the only missing ingredient in that setting.

\paragraph{The rate of escape}
Let us consider exponents often used to measure the speed
of the random walk.
The first is the {\em speed exponent} associated to the mean-squared displacement:  A value $\beta > 0$ such that
\[
   \E\left[\dist_G(X_0,X_T)^2 \mid X_0=\bm{0}\right] \asymp T^{2/\beta + o(1)} \quad \textrm{as}\quad T \to \infty.
\]
Note that $\beta = 2$ corresponds to standard diffusive speed.

The second is often called the {\em walk dimension}:  For $R \geq 0$, denote
the random time
\[\tau_R \seteq \min \{ t \geq 0 : \dist_G(X_0, X_t) \geq R \},\]
and suppose there
is a value $d_w$ such that
\[
   \E[\tau_R \mid X_0=\bm{0}] \asymp R^{d_w+o(1)} \quad \textrm{as} \quad R \to \infty.
\]
For concreteness, we define
\begin{align*}
   \underline{d}_w &\seteq \liminf_{R \to \infty} \frac{\log \E[\tau_R]}{\log R}, \\
   \underline{\beta} &\seteq \liminf_{t \to \infty} \frac{2 \log t}{\log \E[\dist_G(X_0,X_t)^2]} \\
   \underline{\beta}^* &\seteq \liminf_{n \to \infty}\frac{2 \log n}{\log \E[\max_{1 \leq t \leq n} \dist_G(X_0,X_t)^2]}.
\end{align*}
While $\underline{d}_w$ and $\underline{\beta}$ are incomparable in general, it holds that
$\min\left(\underline{\beta},\underline{d}_w\right) \geq \underline{\beta}^*$.
To see that $\underline{d}_w \geq \underline{\beta}^*$, denote $\cM_n \seteq \max_{1 \leq t \leq n} d_G(X_0,X_t)$, and note that
\[
   \E\left[\cM_n^2\right] \geq \Pr(\tau_R \leq n) R^2 \geq \1_{\{\E[\tau_R] \leq n/2\}} \frac12 R^2,
\]
hence if $\E[\cM_n^2] \leq n^{2/\underline{\beta}^*+o(1)}$ as $n \to \infty$, then
$\E[\tau_R] > R^{\underline{\beta}^*-o(1)}$ as $R \to \infty$, implying $\underline{d}_w \geq \underline{\beta}^*$.

Our main result is the strict inequality $\underline{\beta}^* > 2$ for the IIC in dimension two.
In fact it is interesting to bound these quantities
explicitly in terms of the arm probabilities defined as $\pi_1(n)$
and $\pi_2(n)$, respectively.
Let $ \eta_1,\eta_2, \eta_{21} > 0$ be such that, as $n \to \infty$,
\begin{align*}
   \pi_1(n) &\leq n^{-\eta_1+o(1)} \\
   \pi_2(n) &\leq n^{-\eta_{21}+o(1)} \pi_1(n)\le n^{-\eta_2+o(1)}.
\end{align*}
Note that by \pref{thm:kesten-arm}, one can take $\eta_{21}\ge \eta_1.$
Further, if $\pi_1(n) \asymp n^{-\eta_1+o(1)}$ and $\pi_2(n) \asymp n^{-\eta_2+o(1)}$,
we could take $\eta_{21} = \eta_2 - \eta_1$.

\begin{theorem}\label{thm:speed-iic}
   For every $\delta$ satisfying
   \[
      \delta < \frac{\eta_{21}}{2-\eta_1},
   \]
   it holds that
   \[
      \E\left[\max_{1 \leq t \leq n} \dist_G(X_0,X_t)^2\right] \leq n^{1/(1+ \delta/2)+o(1)}.
   \]
\end{theorem} 
\pref{thm:kesten-arm} gives
a choice of $\eta_{21},\eta_1 > 0$ such that
\[
   \frac{\eta_{21}}{2-\eta_1} > \frac{\eta_{21}}{2} > 0,
\]
showing that the speed of random walk is indeed subdiffusive.

Lower bounds have been given previously for the Euclidean analog $\underline{d}_w^{\mathrm{Euc}}$,
where one considers the random time $\tau^{\mathrm{Euc}}_R \seteq \min \{ t \geq 0 : \|X_t-X_0\|_2 \geq R \}$.
In \cite{DHP13}, it is reported that Kesten's argument yields
\[
   \underline{d}_w^{\mathrm{Euc}} \geq 2 + \frac{\eta_1^2}{4}\,,
\]
and this estimate is improved (see \cite[Rem. 1]{DHP13}) to $\underline{d}_w^{\mathrm{Euc}}\ge 2+\frac{1}{2}\eta_1\eta_2,$  and further, under the assumption of the existence of the one and two arm exponents, to 
\begin{equation}\label{eq:dhp}
   \underline{d}_w^{\mathrm{Euc}} > 2+\frac{\eta_2\eta_{21}}{2}\,.
\end{equation}
Even for the Euclidean exponent, our argument yields the improved bound
\begin{equation}\label{eq:dw-lb}
   \underline{d}_w^{\mathrm{Euc}} \geq \underline{d}_w \geq \underline{\beta}^* \geq 2+\frac{\eta_{21}}{2-\eta_1}> 2+\frac{\eta_2\eta_{21}}{2}\ge 2+\frac{\eta_1\eta_2 }{2}.
\end{equation}
Here the fourth inequality follows since $\eta_2\le 1$ and $\eta_1>0$ by \pref{thm:kesten-arm}.
Further for comparison, using the conjectured values $\eta_1 = 5/48$ and $\eta_2 = \eta_1+1/4$,
the lower bound in \eqref{eq:dhp} is $2+17/384 \approx 2.044$, whereas the improved bound in \eqref{eq:dw-lb} gives $2+12/91\approx 2.132$. 

\pref{thm:speed-iic} is in fact a consequence of the much more general \pref{thm:thick-thin} we state later.

\subsection{Preliminaries}

We will consider primarily connected, undirected graphs $G=(V,E)$,
which we equip with the associated path metric $\dist_G$.
We will write $V(G)$ and $E(G)$, respectively, for the vertex and edge
sets of $G$.
For $v \in V$, let $\deg_G(v)$ denote the degree of $v$ in $G$.
For $v \in V$ and $r \geq 0$, we use $B_G(v,r) = \{ u \in V : d_G(u,v) \leq r\}$
to denote the closed ball in $G$.
For a subset $S \subseteq V(G)$, we write $G[S]$ for the subgraph induced on $S$.
For two subsets $S,T \subseteq V(G)$, write $S \conn{G} T$ if there is
a path in $G$ from $S$ to $T$.
For a subset $S \subseteq V(G)$, define the (inner) vertex boundary
\[
   \partial_G S \seteq \{ v \in S : \exists \{u,v\} \in E(G), u \notin S \}.
\]

If $\omega : V(G) \to \R_+$, we define a length functional $\len_{\omega}$
on edges $\{x,y\} \in E(G)$ via $\len_{\omega}(\{x,y\}) \seteq (\omega(x)+\omega(y))/2$,
and extend this additively to arbitrary paths $\gamma$ in $G$:
\[
   \len_{\omega}(\gamma) \seteq \sum_{e \in \gamma} \len_{\omega}(e)\,.
\]
Finally, for $x,y \in V(G)$, define
\[
   \dist_{\omega}(x,y) \seteq \inf \left\{ \len_{\omega}(\gamma) : \gamma \textrm{ is an $x$-$y$ path in $G$} \right\}.
\]

For two expressions $A$ and $B$, we use the notation $A \lesssim B$ to denote
that $A \leq C B$ for some {\em universal} constant $C$.
We write $A \asymp B$ for the conjunction $A \lesssim B \wedge B \lesssim A$.

\subsection{Unimodular random networks}

It will be important for us to envision the IIC as a unimodular random graph.
Let $\cG_{\bullet}$ denote the set of isomorphism classes of locally-finite, rooted graphs.
A {\em (rooted) network} is a pair $(G,\rho) \in \cG_{\bullet}$,
along with a set of marks $\Psi : V(G) \cup E(G) \to \Xi$, where $\Xi$ is a complete, separable metric space.
A {\em unimodular random network} is a random triple $(G,\rho,\Psi)$ that satisfies the following mass-transport principle.
For every nonnegative, automorphism-invariant Borel functional $F(G,x,y,\Psi)$ on doubly-rooted, marked networks:
\[
   \E\left[\sum_{x \in V} F(G,\rho,x,\Psi)\right] = \E\left[\sum_{x \in V} F(G,x,\rho,\Psi)\right].
\]
By {\em automorphism-invariant}, we mean that for any automorphism $\f$ of $G$, it holds that
$F(\f(G),\f(x),\f(y),\Psi \circ \f^{-1}) = F(G,x,y,\Psi)$.
We refer to the extensive monograph \cite{aldous-lyons}.

A {\em stationary random network} is a random triple $(G,X_0,\Psi)$, where $X_0 \in V(G)$,
and such that $(G,X_0,\Psi)$ and $(G,X_1,\Psi)$ have the same law, where $\{X_n\}$ is the random walk on $G$.
A {\em reversible random network} is a stationary random network in which $(G,X_0,X_1,\Psi)$
and $(G,X_1,X_0,\Psi)$ have the same law.
In some cases, these three notions are related.

The following lemma is due to R. Lyons (see \cite[Thm 3.3]{glp17}); it asserts that stationary random networks of subexponential growth are reversible.
\begin{lemma}
   \label{lem:lyons}
   If $(G,X_0, \Psi)$ is a stationary random network such that
   \[
      \lim_{n \to \infty} \frac{\E[\log |B_G(X_0,n)|]}{n} = 0,
   \]
   then $(G,X_0,\Psi)$ is reversible.
\end{lemma}
When $\E[\deg_G(\rho)] < \infty$, unimodular random graphs and reversible random graphs are equivalent,
up to biasing the measure by the degree of the root.
\begin{lemma}[\cite{bc12}]
   \label{lem:bc12}
   Consider a random network $(G,\rho,\Psi)$ with $\E[\deg_G(\rho)] < \infty$.
   Then $(G,\rho,\Psi)$ is unimodular if and only if the network $(\tilde{G},\tilde{\rho},\tilde{\Psi})$ is reversible,
   where the law of $(\tilde{G},\tilde{\rho},\tilde{\Psi})$ is that of $(G,\rho,\Psi)$ biased by $\deg_G(\rho)$.
\end{lemma}

\begin{remark}Although \cite[Thm 3.3]{glp17} and \cite [Prop. 2.5]{bc12} only state the above result only for random rooted graphs, the proofs works verbatim for networks. 
\end{remark}

In particular, \pref{lem:bc12} provides an equivalent mass-tranport principle for reversible random networks $(G,X_0,\Psi)$:
For every nonnegative, automorphism-invariant Borel functional $F(G,x,y,\Psi)$:
\begin{equation}\label{eq:mtp0}
      \E\left[\frac{1}{\deg_G(X_0)} \sum_{y \in V(G)} F(G,X_0,y,\Psi)\right]
      =
      \E\left[\frac{1}{\deg_G(X_0)} \sum_{y \in V(G)} F(G,y,X_0,\Psi)\right].
\end{equation}

\section{Subdiffusivity of the random walk}

\subsection{The thin backbone condition}
\label{sec:thin-backbone}

We first introduce some notions that will allow us to
quantify the condition that paths at a given scale
travel along an asymptotically small set of vertices (the ``backbone'').

\begin{definition}[Covering systems]
   \label{def:csystem}
Consider a connected graph $\base$
and a collection of finite subsets $\cC$ covering $V(\base)$.
Say that $\cC$ is {\em $\Delta$-bounded} if $\diam_{\base}(S) \leq \Delta$ for every $S \in \cC$.
Say that $\cC$ is {\em $\alpha$-padded} if, for every $v \in V(\base)$, there is some
$S \in \cC$ such that $B_{\base}(v,\alpha) \subseteq S$.
Denote by
\[
   \mult(\cC) \seteq \sup \left\{ \# \{ S' \in \cC : S \cap S' \neq \emptyset \} : S \in \cC \right\}
\]
the {\em intersection multiplicity of $\cC$.}
In other words, $\mult(\cC)$ is the smallest number $m$ such that every set $S \in \cC$
intersects at most $m$ other sets in $\cC$.

A {\em covering system of $\base$ (with scale parameter $M > 1$)} is a collection
$\mathscr{C} = \llangle \cC^{(k)} : k \geq 1 \rrangle$ of coverings of $V(\base)$ such that 
\begin{enumerate}
   \item $\sup \{ \mult(\cC^{(k)}) : k \geq 1 \} < \infty$
   \item For every $k \geq 1$, $\cC^{(k)}$ is $M^k$-bounded.
\end{enumerate}
Say that a covering system $\sC$ is {\em uniformly $\e$-padded} if $\cC^{(k)}$ is $\e M^k$-padded for each $k \geq 1$.
\end{definition}

\begin{example}\label{example:z2}
As a motivating example,
consider the covering system for $\mathbb{Z}^2$ that we employ in \pref{sec:iic}:
For every integer $k \geq 1$, define the covering
\[
   \cC^{(k)} \seteq \left\{\left((\sigma_1,\sigma_2) 2^{k-2} + \left[i 2^{k-1}, (i+1)2^{k-1}\right] \times \left[j 2^{k-1},(j+1)2^{k-1}\right] \cap \Z^2\right) : i,j \in \Z,
   \sigma_1,\sigma_2 \in \{0,1\}\right\},
\]
and take $\sC \seteq \llangle \cC^{(k)} : k \geq 1 \rrangle$.
One can check that each $\cC^{(k)}$ is $2^k$-bounded and $\mult(\cC^{(k)}) \leq 10$, and moreover
the system $\sC$ is uniformly $1/4$-padded.
\end{example}

We will work with random rooted subgraphs of a given base graph $\base$, as well as a random covering of the latter with bounded intersection multiplicity. We next introduce a marking of $\base$ which encodes all the above.

\begin{definition}[Markings from coverings]
   \label{def:marking}
   Let $\{U_v : v \in V(\base) \}$ be a family of i.i.d. uniform $[0,1]$ random variables.
Now for any  cover $\cC$ of $V(\base)$ with $\mult=\mult(\cC) < \infty$,
 for every $S\in C,$ let $$U_S=\sum_{v\in S}U_v.$$  

Thus, almost surely, $U_{S}$ is distinct for each distinct finite subset $S\in \cC$.
 Now for any $v$, let $U_{v,1}, U_{v,2},U_{v,3}\ldots U_{v,\mult}$, be the values $\{U_S:S\in \cC, v\in S\}$ arranged in non-decreasing order,
 with $U_{v,j}=1$ if $v$ occurs in fewer than $j$ sets in $\cC$.
Define now a marking $\psi_{\cC}$ on $V(\base)$ by
\begin{align*}
      \psi_{\cC}(v) &\seteq \llangle U_{v,1}, U_{v,2}, \ldots U_{v,\mult}\rrangle \text{ for } v \in V(\base). 
\end{align*}
By the distinctness of $U_{S}$ mentioned above, one can reconstruct the covering $\cC$ from  $\psi_{\cC}$ almost surely.

Suppose now that $G$ is a subgraph of $V(\base)$ with $\xi = \xi_G : V(\base) \cup E(\base) \to \{0,1\}$ defined
by $\xi^{-1}(1) = V(G) \cup E(G)$.
For a covering system $\sC = \llangle \cC^{(k)} : k \geq 1 \rrangle$,
 we define a marking $\Psi_{G,\sC}$ on $V(\base)\cup E(\base)$ by 
\begin{align*}
      \Psi_{G,\sC}(v) \seteq \llangle \xi(v), \psi_{\cC^{(k)}}(v) : k \geq 1 \rrangle \text{ for } v \in V(G), \text{ and }
      \Psi_{G,\sC}(e)\seteq  \xi(e)  \text{ for } e \in E(\base).
\end{align*}
Note that for any $k,$ the marking $\psi_{\cC^{(k)}}$ is constructed using the same set of random variables $\{U_v : v \in V(\base) \}$.
By the above discussion, the subgraph $G$ and the covering system $\sC$ are reconstructible from $\Psi_{G,\sC}$.
\end{definition}

\paragraph{Backbones}
For any $S\subset V(\base),$ define  the {\em backbone} $\mathsf{B}_G[S] \subseteq S$ as
the set of vertices $v \in S$ such that there are paths $\gamma,\gamma'$
in $G[S]$
with $\gamma \cap \gamma' = \{v\}$,
and each of $\gamma$ and $\gamma'$
connect $v$ to $V(G) \setminus S$,
where we abuse notation to denote by $G[S]$ the subgraph induced by $G$ on $V(G)\cap S.$

\begin{figure}[h]
\centering
\includegraphics[width=8cm]{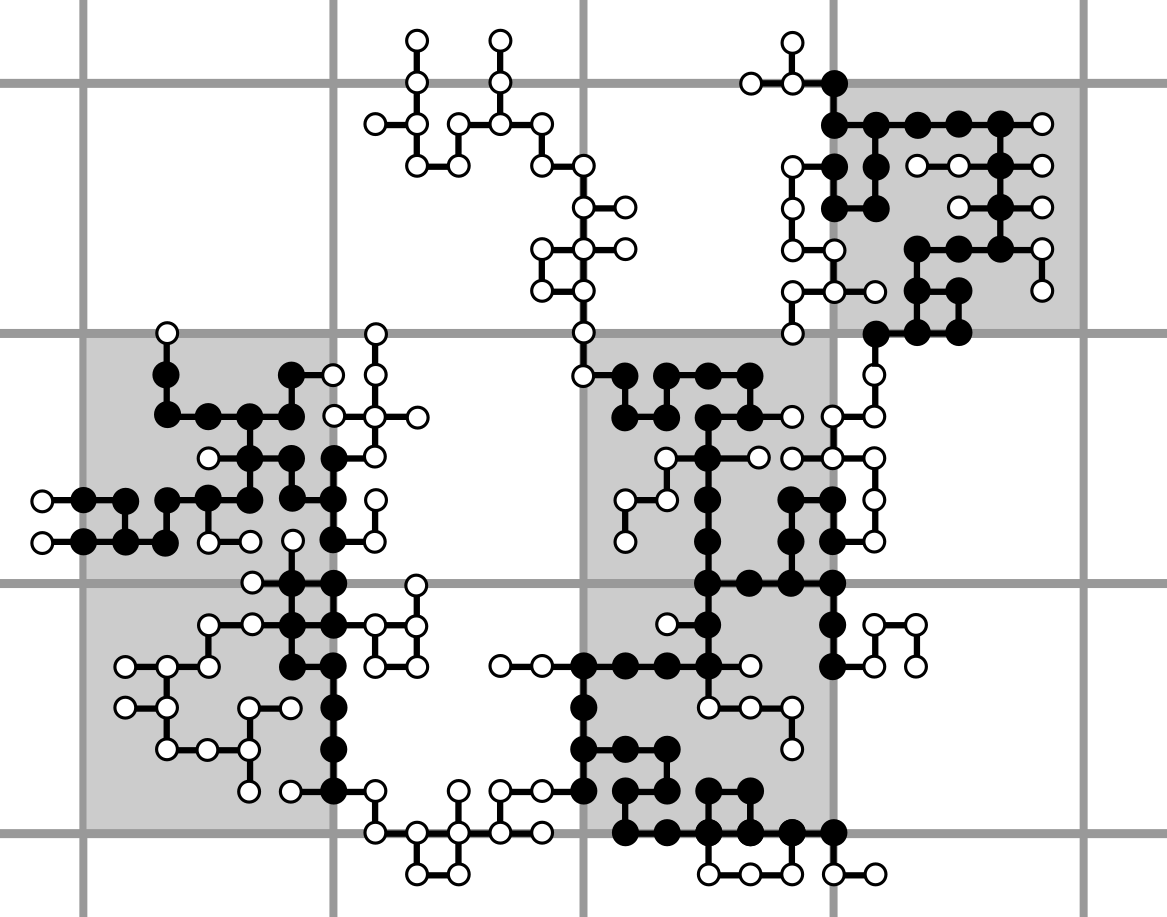}
\caption{A collection of patches.  Deep patches are colored grey, and the backbone
vertices of deep patches are colored black.
\pref{ass:system}(3) asserts that the black vertices are an asymptotically
small fraction of all vertices as $k \to \infty$.\label{fig:patches}}
\end{figure}

\begin{assumption}\label{ass:system}   
   Consider a fixed based graph $\base$, together
   with a random rooted subgraph $(G,X_0)$ of $\base$, and
   a random covering system $\sC$ of $\base$.
   Recall the associated marking $\Psi_{G,\sC}$ from \pref{def:marking}.
   Letting $\{X_n\}$ denote the random walk on $G$,
   we assume that $(\base,X_0,\Psi_{G,\sC})$ is stationary in the sense that
   \begin{equation}\label{eq:base-stationary}
      \left(\base,X_0,\Psi_{G,\sC}\right) \stackrel{\mathrm{law}}{=} 
      \left(\base,X_1,\Psi_{G,\sC}\right).
   \end{equation}
   We record now a straightforward fact.

\begin{fact}[Stationary projection]
   \label{fact:projection}
   If \eqref{eq:base-stationary} holds, then
   for any deterministic, automorphism-invariant,
   real-valued function $\Phi$ on rooted networks, with the marking on the latter living in the same space as  $\Psi_{G,\sC}$,  
the marking $\widehat \Psi$ obtained on $V(G),$ by setting $\widehat\Psi(v) \seteq \Phi((\base, v, \Psi_{G,\sC})),$ makes $(G,X_0, \widehat \Psi)$ into a stationary random network.    
\end{fact}
Additionally, we assume the following conditions.

\begin{enumerate}
   \item {\bf Polynomial volume growth.}
      There are constants $C_1,d \geq 1$ such that
      \begin{equation}\label{eq:growth}
         |B_{\base}(x,r)| \leq C_1 r^{d}, \qquad \forall x \in V(\base), \forall r \geq 1.
      \end{equation}
      Since $G$ is a subgraph of $\base$, we have $|B_G(x,r)| \leq |B_{\base}(x,r)|$.  
      In particular, \pref{lem:lyons} implies that $(G,X_0,\widehat \Psi)$ is a reversible random network for any marking $\widehat \Psi$ such that $(G,X_0,\widehat \Psi)$ is stationary. 

      Thus \eqref{eq:mtp0} yields the following mass transport principle: 
      For any nonnegative, automorphism-invariant Borel functional $F(G,x,y,\widehat \Psi)$:
      \begin{equation}\label{eq:mtp2}
         \E\left[\frac{1}{\deg_G(X_0)} \sum_{y \in V(G)} F(G,X_0,y,\widehat \Psi)\right] = \E\left[\frac{1}{\deg_G(X_0)} \sum_{y \in V(G)} F(G,y,X_0,\widehat \Psi)\right].
      \end{equation}

   \item {\bf Padded covering system.}
      For some $\e > 0$, $\sC$ is almost surely a uniformly $\e$-padded covering system for $\base$.
   \item {\bf Deep patches have thin backbones.}
For a subset $S \subseteq V(\base)$, recalling that $G[S]$ denotes the induced subgraph of $G$ on $V(G)\cap S,$ define the {\em depth of $S$ in $G$} as the quantity
\begin{equation}\label{eq:depth}
   \tau_G(S) \seteq \max \{ \tau \geq 0 : \exists v \in S \textrm{ with } B_{\base}(v,\tau)
   \subseteq S \textrm{ and } v \overset{G[S]}{\leftrightarrow} \partial_G S\}.
\end{equation}

Now given the random covering system $\sC = \llangle \cC^{(k)} : k \geq 1 \rrangle$,
let us denote by 
\[
   {\cC}^{(k)}_{\circ} \seteq \{ S \in \cC^{(k)} : \tau_G(S) \geq \e M^k \}
\]
      the collection of deep patches, and by
      \begin{align}\label{backboneunion}
         {\sfB}^{(k)}_{\circ} &\seteq \bigcup_{S \in {\cC}^{(k)}_{\circ}} \sfB_G[S]
      \end{align}
      the union of their backbones.  (See \pref{fig:patches}
      for an illustration.)
      We assume that for some $\eta > 0$:
      \[
         \Pr\left[X_0 \in {\sfB}^{(k)}_{\circ}\right] \leq C_4 M^{-\eta k},
      \]
      where we recall $M$ from \pref{def:csystem}.
\end{enumerate}
\end{assumption}

Recall that $\Psi_{G,\sC}$ allows one to reconstruct $\cC^{(k)}$ for any $k,$ as well as the random subgraph $G.$
Thus by \pref{fact:projection}, the following consequence holds.

\begin{lemma}\label{lem:isrev}
   If \pref{ass:system} is satisfied, then $\left(G,X_0,\1_{\sfB^{(k)}_{\circ}}\right)$ is a reversible random network for each $k \geq 1$.
\end{lemma}

We now state our main general theorem proving subdiffusivity of random walk under the above assumptions.

\begin{theorem}\label{thm:thick-thin}
   If $(G,X_0)$ satisfies \pref{ass:system},
   then the random walk is subdiffusive:  For some constant $C \geq 1$,
   \[
      \E\left[\max_{0 \leq t \leq T} \dist_G(X_0, X_t)^2\right] \leq C T^{1/(1+\eta/2d)} (\log T)^4, \qquad \forall T \geq 2.
   \]
\end{theorem}
As alluded to before, to prove the theorem we rely on the framework developed in \cite{lee17a} which we review next.

\subsection{Conformal weights on stationary random graphs}

Consider a stationary random graph $(G,X_0)$, and let $\{X_t\}$ denote the random walk on $G$.
A {\em conformal weight on $(G,X_0)$} is a random mapping $\omega : V(G) \to \R_+$ such that $(G,X_0,\omega)$
and $(G,X_1,\omega)$ have the same law, and $\E[\omega(X_0)^2] < \infty$.
For establishing subdiffusivity of the IIC in two dimensions,
the following theorem suffices.

\begin{theorem}[\cite{lee17a}]
   \label{thm:mt}
   Suppose that $(G,X_0)$ is a planar, invariantly amenable\footnote{One can find the notion of {\em invariantly amenability} in \cite[Sec 3.2]{ahnr18},
      where it is imported from \cite{aldous-lyons}.
      We will only consider stationary random graphs with an almost sure polynomial upper bound
   on their volume growth, which are easily seen to be invariantly amenable; see, e.g., \cite[Lem. 4.3]{lee17a}.}, stationary random graph.
   Then for any conformal weight $\omega$ on $(G,X_0)$, it holds that for all $T \geq 1$,
   \[
      \E\left[\max_{0 \leq t \leq T} \dist_{\omega}(X_0,X_t)^2\right] \lesssim T \E[\omega(X_0)^2].
   \]
\end{theorem}

   Note that \cite{lee17a} contains the weaker claim
   \[
      \E\left[\dist_{\omega}(X_0,X_T)^2\right] \lesssim T \E[\omega(X_0)^2].
   \]
   But the latter conclusion is derived
   from the fact that weighted planar graph metrics have Markov type
   2 with a uniform constant \cite{dlp13}.
   As observed in \cite{GH20}, the authors of \cite{dlp13} actually bound
   the {\em maximal} Markov type 2 constant, as do all methods
   that establish Markov type bounds using the forward/backward martingale 
   decomposition \cite{npss06}.
   The same observation applies to the next theorem,
   since its proof uses only 
   that the Markov type 2 constant for an $n$-point metric
   space is $O(\log n)$, which also holds for the maximal Markov type 2 constant.
   We give a more detailed explanation of how to obtain these strengthenings
   in \pref{rem:maximal-mt} below.

\begin{theorem}[{\cite[Thm 5.7]{lee17a}}]
   \label{thm:mt2}
   Suppose $(G,X_0)$ is a stationary random graph that
   satisfies, for some $q \geq 1$,
   \begin{equation}\label{eq:annealed-growth}
      \E |B_G(X_0,r)| \lesssim r^q\qquad \forall r \geq 1\,.
   \end{equation}
   Then for any conformal weight $\omega$ on $(G,X_0)$, and any
   $q' \geq 1$, there is a constant $C=C(q,q')$ such that for all $T \geq 2$,
   \[
      \E\left[T^{q'} \wedge \max_{0 \leq t \leq T} \dist_{\omega}(X_0,X_t)^2\right] \leq 1 + C T (\log T)^2 \E[\omega(X_0)^2].
   \]
\end{theorem}

The next corollary makes precise the strategy of showing subdiffusivity of the random walk by appropriate choice of the conformal weights.

\begin{corollary}\label{cor:conformal}
   Assume $(G,X_0)$ is a stationary random graph satisfying \eqref{eq:annealed-growth}
   for some $q \geq 1$.  Then there is a constant $C=C(q)$ such that the following holds.
   Suppose that for some $\delta > 0$ and
   every $\tau > 0$, there exists a conformal weight $\omega_{\tau}$ on $(G,X_0)$
   such that $\E[\omega_{\tau}(X_0)^2] \leq 1$ and for all $x,y \in V(G)$,
   \[
      \dist_G(x,y) \geq \tau \implies \dist_{\omega_{\tau}}(x,y) \gtrsim \tau^{1+\delta}.
   \]
   Then for all $T \geq 2$,
   \[
      \E\left[\max_{0 \leq t \leq T} \dist_G(X_0,X_T)^2\right] \leq C T^{1/(1+\delta)} (\log T)^4.
   \]
\end{corollary}

\begin{proof}
   Define the weight
   \[
      \omega \seteq \left(\frac{6}{\pi^2} \sum_{j \geq 1} \frac{\omega^2_{2^j}}{j^2}\right)^{1/2},
   \]
   so that $\E[\omega(X_0)^2] \leq \frac{6}{\pi^2} \sum_{j \geq 1} j^{-2} = 1$.

   Next, consider $x,y \in V(G)$ such that $\dist_G(x,y) \in [2^j, 2^{j+1})$ and $j \geq 0$.  Then:
   \[
      \dist_{\omega}(x,y) \gtrsim \frac{\dist_{\omega_{2^j}}(x,y)}{j} \gtrsim \frac{2^{j(1+\delta)}}{j} \gtrsim \frac{\dist_G(x,y)^{1+\delta}}{\log \dist_G(x,y)}.
   \]
   Applying \pref{thm:mt2} gives
   \begin{align*}
      \E\left[\max_{0 \leq t \leq T} 
         \frac{\dist_G(X_0,X_t)^{2(1+\delta)}}{\left(\log \dist_G(X_0,X_t)\right)^{2}}
   \right] &=
      \E\left[T^{2(1+\delta)} \wedge \max_{0 \leq t \leq T} 
         \frac{\dist_G(X_0,X_t)^{2(1+\delta)}}{\left(\log \dist_G(X_0,X_t)\right)^{2}}
   \right] \\
&\lesssim \E\left[T^{2(1+\delta)} \wedge \max_{0 \leq t \leq T} \dist_{\omega}(X_0,X_t)^2\right] \\
      & \leq C' T (\log T)^2,
   \end{align*}
   for some number $C' = C'(q)$.
   The desired bound follows using $\dist_G(X_0,X_T) \leq T$ and H\"{o}lder's inequality.
\end{proof}

\begin{remark}[Maximal Markov type]\label{rem:maximal-mt} Although alluded to multiple times already, we now formally introduce the notion of Markov type followed by the strengthened notion of Maximal Markov type.
   
A metric space $(X,d)$ is said to have {\em Markov type $p \in [1,\infty)$} if there is a constant $M > 0$ such that for every $n \in \mathbb N$, the following holds. For every reversible Markov chain $\{Z_t\}_{t=0}^{\infty}$ on $\{1,\ldots,n\}$, every mapping $f : \{1,\ldots,n\} \to X$, and every time $t \in \mathbb N$, 
   \begin{equation}\label{eq:mtype} \E \left[d(f(Z_t), f(Z_0))^p\right] \leq M^p t \, \E \left[d(f(Z_0), f(Z_1))^p\right]\,, \end{equation}
   where $Z_0$ is distributed according to the stationary measure of the chain. One denotes by $M_p(X,d)$ the infimal constant $M$ such that the inequality holds.
   The space $(X,d)$ has {\em maximal Markov type $p \in [1,\infty)$} if the following strengthening of \eqref{eq:mtype} holds:
   \[
      \E \left[\max_{1 \leq s \leq t} d(f(Z_s), f(Z_0))^p\right] \leq M^p t \, \E \left[d(f(Z_0), f(Z_1))^p\right].
   \]
   Let us write $M_p^{\max}(X,d)$ for the analogous infimal constant.
   The following theorem is a restatement of \cite[Thm. 5.2]{lee17a}.

   \begin{theorem}\label{thm:lee17a-52}
      Suppose that $(G,X_0)$ is an invariantly amenable reversible random graph.  Then for any conformal metric $\omega$ on $(G,X_0)$, the following holds.
      If there is a number $M$ such that $M_2(V(G),\dist_{\omega}) \leq M$ holds almost surely, then for every $T \geq 1$,
      \[
         \E[\dist_{\omega}(X_0,X_T)^2] \leq M T \E[\omega(X_0)^2].
      \]
   \end{theorem}

   We claim that the same argument gives an analogous theorem for maximal Markov type:

   \begin{theorem}
      Suppose that $(G,X_0)$ is an invariantly amenable reversible random graph.  Then for any conformal metric $\omega$ on $(G,X_0)$, the following holds.
      If there is a number $M$ such that $M^{\max}_2(V(G),\dist_{\omega}) \leq M$ holds almost surely, then for every $T \geq 1$,
      \[
         \E\left[\max_{1 \leq t \leq T} \dist_{\omega}(X_0,X_T)^2\right] \leq M T \E[\omega(X_0)^2].
      \]
   \end{theorem}

   This can be obtained directly
   from the proof of Theorem 5.2 in \cite{lee17a} by adding $\max_{1 \leq t \leq T}$ to the left-hand side of the first inequality
   in that proof, and then using this stronger inequality in the 3rd, 4th, and 5th inequalities of that proof.

   The same sort of substitution allows one to modify the proof of \cite[Thm 5.7]{lee17a} to obtain \pref{thm:mt2},
   with one additional observation.  Equation (5.10) in the proof of that theorem employs \cite[Thm 5.3]{lee17a} which
   asserts that for any $n$-point metric space $(X,d)$ with $n \geq 2$, it holds that $M_2(X,d) = O(\log n)$.
   Here one needs to replace this by the stronger fact that $M_2^{\max}(X,d) = O(\log n)$.

   This fact is well-known and follows from two observations:  By Bourgain's embedding theorem \cite{Bourgain85},
   every $n$-point metric space admits an embedding into $\ell_2$ with $O(\log n)$ bilipschitz distortion.
   From the definition, this immediately implies that $M_2^{\max}(X,d) = O(\log n) M_2^{\max}(\ell_2)$, where
   $M_2^{\max}(\ell_2)$ is the maximal Markov type $2$ constant of the separable Hilbert space $\ell_2$.
   And it is known (see, e.g., \cite{npss06}) that $M_2^{\max}(\ell_2) = O(1)$.

   Finally, let us remark that if a stationary random graph $(G,X_0)$ is invariantly amenable,
   then it is automatically also a reversible random graph.
   This is because if $(G,X_0)$ is invariantly amenable, then it is the distributional limit
   of a sequence of finite stationary random graphs, and all such graphs are reversible random graphs.
\end{remark}

\subsection{Proof of \pref{thm:thick-thin}}
\label{sec:thick-thin}

Fix $k \geq 1$ and
denote $p_k \seteq \Pr[X_0 \in {\sfB}^{(k)}_{\circ}]$ where ${\sfB}^{(k)}_{\circ}$ was defined in \eqref{backboneunion}, and
$\omega \seteq \1_{{\sfB}^{(k)}_{\circ}}/\sqrt{p_k}$.
Since $(G,X_0,\Psi_{G,\sC})$ and $(G,X_1,\Psi_{G,\sC})$ have the same law, so do $(G,X_0,\omega)$ and $(G,X_1,\omega)$
(recall \pref{lem:isrev}).
Furthermore, by definition $\E[\omega(X_0)^2] =1.$

Define $D \seteq 3 \lceil C_1 M^{kd}\rceil$,
and consider any pair $x,y \in V(G)$ with $\dist_G(x,y) \geq D$.
Let $\gamma$ be a {\em simple} path in $G$ connecting $x$ and $y$ 
and denote
\[
   Z \seteq \{ z \in V(\gamma) : \dist_{G}(z, \{x,y\}) > D/3 \}.
\]
Since by hypothesis $\dist_G(x,y) \geq D$, it follows that $|Z|\ge D/3.$ Now 
for any $z\in Z$ consider the paths $\gamma_1$ and $\gamma_2$ which are obtained by splitting $\gamma$ at $z.$ Thus $\gamma_1$ and $\gamma_2$ are disjoint simple paths from $x$ to $z$ and  $z$ to $y$ respectively.

Now, note that since $\sC$ is a uniformly $\e$-padded covering system there is almost surely a set
$S_z \in \cC^{(k)}$ for which $B_{\base}(z,\e M^k) \subseteq S_z$.
Moreover, since $\diam_{\base}(S_z) \leq M^k$, we have $|S_{z}|\le C_1 M^{kd}$.
Now since $|\gamma_1|, |\gamma_2|\ge D/3=C_1 M^{kd},$ it follows that both $\gamma_1$ and $\gamma_2$ intersect $S_z^c,$ implying  $z \in \sfB_G[S_z]$.
In particular, we have $z \in {\sfB}^{(k)}_{\circ}$.

Therefore,
\[
   \len_{\omega}(\gamma) \geq p_k^{-1/2} |Z| \gtrsim \sqrt{\frac{M^{\eta k}}{C_3}} \dist_G(x,y)
   \gtrsim \frac{\dist_G(x,y)^{1+\eta/2d}}{C_1^{\eta/2d} C_3^{1/2}},
\] where we used $p_{k}\le C_4 M^{-\eta k}$ by \pref{ass:system}(3). 
Since this holds for any $x$-$y$ path $\gamma$ in $G$, we have
\begin{equation}
   \label{eq:dist-lb}
   \dist_{\omega}(x,y) \gtrsim \frac{\dist_G(x,y)^{1+\eta/2d}}{C_1^{\eta/2d} C_3^{1/2}}.
\end{equation}
Now applying \pref{cor:conformal} with $\delta=\eta/2d$ completes the proof, recalling that $(G,X_0)$ is reversible
by \pref{ass:system}(1).

\subsection{Intrinsic volume growth}

We will improve the speed bound in \pref{thm:thick-thin}
by assuming a high-probability estimate
for the volume of patches in the random subgraph $G$.

\begin{assumption}\label{ass:system2}
   Consider \pref{ass:system} and, in addition, suppose that for some $d' < d$ 
   and number $c_4 > 0$:
\begin{enumerate}
   \item[4.] {\bf Annealed volume estimate.}
      For $k \geq 1$, define
      \begin{align}\label{largepatch}
      \cC^{(k)}_{\bullet} &\seteq \left\{ S \in \cC^{(k)} : |S \cap V(G)| \ge c_4 M^{kd'} \right\} \\
      \nonumber
      \cL^{(k)} &\seteq \bigcup_{S \in \cC^{(k)}_{\bullet}} S.
      \end{align}
      Thus $\cC^{(k)}_{\bullet}$ denotes the set of patches in $\cC^{(k)}$ which have a large intersection with the random subgraph $G$, and $\cL^{(k)}$ denotes the union of all the patches in $\cC^{(k)}_{\bullet}.$ 
      Our assumption states that large patches are rare, i.e., the probability the root falls in a patch of large cardinality is small:
      \begin{equation}\label{eq:big-patch}
         \Pr\left[X_0 \in \cL^{(k)}\right] \leq C_3 M^{-\eta k}.
      \end{equation}
\end{enumerate}
\end{assumption}
\begin{theorem}\label{thm:thick-thin2}
   If $(G,X_0)$ satisfies \pref{ass:system} and \pref{ass:system2},
   then the random walk is subdiffusive:  There is some constant $C \geq 1$ such that
   \[
      \E\left[\max_{0 \leq t \leq T} \dist_G(X_0, X_t)^2\right] \leq C T^{1/(1+\eta/2d')} (\log T)^4, \qquad \forall T \geq 2.
   \]
\end{theorem}
Note the improved exponent compared to \pref{thm:thick-thin}.
\begin{proof} We start by modifying our choice of conformal weight from \pref{sec:thick-thin}. To this end, fix $k \geq 1$.
Define $\omega$ as in \pref{sec:thick-thin} and, furthermore,
\begin{align*}
   q_{k} &\seteq \Pr[X_0 \in \cL^{(k)}] \\
   \omega_1 &\seteq \1_{\cL^{(k)}}/\sqrt{q_{k}} \\
   \hat{\omega} &\seteq \sqrt{(\omega^2 + \omega_1^2)/2}.
\end{align*}
As before, since $(G,X_0,\psi_{G,\cC^{(k)}})$ and $(G,X_1,\psi_{G,\cC^{(k)}})$ have the same law, so do $(G,X_0,\hat{\omega})$ and $(G,X_1,\hat{\omega})$,
as $\omega$ and the set $\cL^{(k)}$ are determined by the marking $\Psi_{G,\sC}$ (recall \pref{lem:isrev}).

Define $D_k \seteq 3 \lceil c_4 M^{k d'}\rceil$, 
and consider any pair $x,y \in V(G)$ with $D_k\leq\dist_G(x,y) \leq D_{k+1}$. {In the rest of the proof we will drop the $k$ dependence and denote $D_k$ by $D.$}
Let $\gamma$ be a simple path connecting $x$ and $y$ in $G$,
and define again
\[
   Z \seteq \{ z \in V(\gamma) : \dist_{G}(z, \{x,y\}) > D/3 \}.
\]
Using the fact that $\sC$ is uniformly $\e$-padded, for each $z \in Z$, let $S_z \in \cC^{(k)}$ be such that $B_{\base}(z, \e M^k) \subseteq S_z$.
Since by hypothesis $\dist_G(x,y) \geq D$, it follows that $|Z|\ge D/3.$ Now as before,
for any $z\in Z,$ consider the paths $\gamma_1$ and $\gamma_2$ which are obtained by splitting $\gamma$ at $z.$ Thus $\gamma_1$ and $\gamma_2$ are disjoint simple paths from $x$ to $z$ and  $z$ to $y$ respectively and further $|\gamma_1|, |\gamma_2|\ge D/3=c_4 M^{k d'}.$

Now, one of two things can happen. Either $|S_z \cap V(G)|< c_4 M^{k d'},$ in which case both $\gamma_1$ and $\gamma_2$ intersect $S_z^c$, implying $z \in \sfB_G[S_z]$.
In particular, we have $z \in {\sfB}^{(k)}_{\circ}$, implying $\omega_1(z) \geq p_{k}^{-1/2}.$ 
Otherwise, $|S_z \cap V(G)|\ge c_4 M^{k d'},$ in which case 
$S_z \in \cC^{(k)}_{\bullet}$ and hence $z\in \cL^{(k)}$, implying  $\omega_1(z) \geq q_{k}^{-1/2}$.

We conclude that
\[
   \omega(z)+\omega_1(z) \geq \min\left(p_{k}^{-1/2},q_{k}^{-1/2}\right),
\]
Now using \pref{ass:system2}(4), this establishes, as in \pref{sec:thick-thin}, that
\[
   \dist_{\hat{\omega}}(x,y) \geq c \dist_G(x,y)^{1+\eta/2d'}
\]
for some number $c > 0$ (independent of $x$ and $y$) where we use the assumed upper bound on $\dist_G(x,y)$.
An application of \pref{cor:conformal} with $\delta=\eta/2d'$ completes the proof.
\end{proof}

\section{The incipient infinite cluster}
\label{sec:iic}

We now prove \pref{thm:speed-iic} establishing subdiffusivity of random walk on the IIC in two dimensions using \pref{thm:thick-thin} and \pref{thm:thick-thin2}.
First, we recall a formulation of the IIC in two dimensions.
Denote $\base \seteq \Z^2$ and $\cS(n) \seteq [-n,n]^2$.
Let $\base^{1/2}$ denote the random subgraph of $\base$ resulting from critical ($p=1/2$)
bond percolation, and let $\hat{G}_n$ denote the largest connected component of the
induced graph $\base^{1/2}[\cS(n)]$.
Choose $\rho_n \in V(\hat G_n)$ uniformly at random.

Let us define $G_n$ as the subgraph of $\Z^2$ that results from translating $\rho_n$
to the origin.  Then J\'arai \cite{Jarai03} has shown that $G_n$ converges weakly to a subgraph $G$ of $\Z^2$
that coincides with Kesten's definition of the IIC.
Since $\rho_n \in V(\hat G_n)$ is chosen uniformly at random, $(G_n,\bm{0})$ is a unimodular random graph (considered
as a random element of $\cG_{\bullet}$), and therefore the weak limit $(G,\bm{0})$ is a unimodular random graph as well
(see, e.g., \cite{BS01} for the simple argument).

Recall the one- and two-arm probabilities defined in \pref{sec:intro} and the statement of \pref{thm:speed-iic}. 
We will start by establishing a slightly weaker bound based on the following theorem.
In conjunction with \pref{thm:thick-thin}, by choosing any $\eta < \eta_{21}$, it yields the statement of \pref{thm:speed-iic}
for $\delta < \eta_{21}/2$.

\begin{theorem}
   \label{thm:iic-system}
   There is a random covering system $\sC$ for $\base$ so that
   $(\base,G,\bm{0},\sC)$ satisfies \pref{ass:system}, and for every $k \geq 1$,
   \[
      \Pr\left[\bm{0} \in {\sfB}^{(k)}_{\circ}\right] \lesssim \frac{\pi_2(2^k)}{\pi_1(2^k)} (\log (2^k))^8\,,
   \]
   where $M=2$ is the scale parameter of $\sC$ (recall \pref{def:csystem}).
\end{theorem}

We employ the covering system $\llangle \cC^{(k)} : k \geq 1\rrangle$ from  Example \pref{example:z2}.
Define the randomly shifted covering $\overline{\cC}^{(k)} = (a_k,b_k) + \cC^{(k)}$
where $(a_k,b_k) \in ([0,2^{k-1}) \times [0,2^{k-1}) \cap \Z^2)$ is chosen uniformly at random and independently across $k$, 
and denote $\sC \seteq \langle \overline{\cC}^{(k)} : k \geq 1 \rangle$.

We have established that $\sC$ is a uniformly $1/4$-padded covering system (with scale parameter $M=2$).
Moreover, by  construction of the random shifts $\{ (a_k,b_k) : k \geq 1\}$, it follows that
\eqref{eq:base-stationary} holds.
We have thus verified a covering system for which \pref{ass:system}(ii) holds.
What remains is to verify \pref{ass:system}(iii).

For $S \subseteq V(\base)$, let us denote ${S}\restrict_G \seteq S \cap V(G).$ Also for brevity we will simply use $\sfB_{G}[S],$ $\tau_{G}[S]$ and $\partial_G S$ to denote $\sfB_{G}[S \restrict_{G}]$, $\tau_{G}[S \restrict_{G}],$ and $\partial_G[ S \restrict_{G}]$ respectively.
Kesten's notion of backbone \cite{Kesten86} is somewhat different from ours, and not unimodular.
For a integer $m \geq 0$, define $\tilde{\sfB}_G(m)$
as the set of vertices $v \in V(G) \cap \cS(m)$ such that there are paths
$\gamma$ and $\gamma'$ connecting $v$ to $\partial_{\base} \cS(m)$ and $\bm{0}$, respectively,
and such that $\gamma \cap \gamma' = \{v\}$.

The next lemma is the key geometric input about the IIC that we will rely on. 
It concerns a fixed square $D^{(q)}$ and a dilated copy $D^{(3q)}$ of $D^{(q)}$;
it asserts an upper bound on the expected size of the backbone with respect to $D^{(q)}$,
and a lower bound on the size of $D^{(3q)} \restrict_G$ whenever there
is a path from $\partial_G D^{(q)}$ to $\partial_G D^{(3q)}$.
See \pref{fig:kesten}.

\begin{figure}[h]
\centering
\includegraphics[width=8cm]{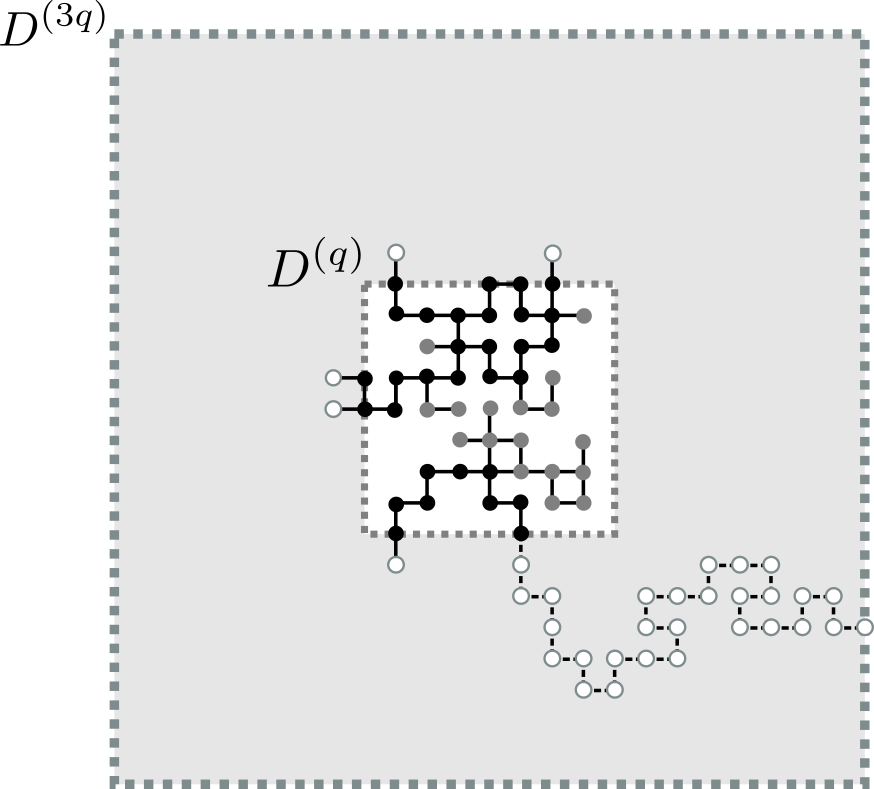}
\caption{$D^{(q)}_{x,y,z}$ inside $D^{(3q)}_{(x-1)/3,(y-1)/3,z}$.
The vertices of $\sfB_G[D^{(q)}_{x,y,z}]$ are colored black.\label{fig:kesten}}
\end{figure}

\begin{lemma}[{\cite[Lem 3.20]{Kesten86}}]
   \label{lem:kesten}
   For $q \in \N$ and $x,y \in \Z, z\in \Z^2$, define
   \begin{align*}
      D_{x,y,z}^{(q)} &\seteq z+[x q, (x+1) q) \times [y q, (y+1) q).
   \end{align*}
   \begin{enumerate}[label=(K\arabic*)]
      \item It holds that for any $m \geq q$,
         \[
            \E\left[|\tilde{\sfB}_{G}(m) \cap D_{x,y,z}^{(q)}]|\right] \lesssim q^2 \pi_2(q).
         \] 
         \label{item:kesten1}
 
      \item There are constants $C,c > 0$ such that for all $\lambda > 0$,
         \[
            \Pr\left(\partial_G D_{x,y,z}^{(q)} \conn{G} \partial_G D_{({x-1})/{3},({y-1})/{3},z}^{(3q)}
            \textrm{ and } \big| D_{(x-1)/3,({y-1})/{3},z}^{(3q)} \restrict_G\big| < \lambda^{-1} q^2 \pi_1(q) \bigmid X_0\right) \leq C \exp\left(-c \lambda^{1/8}\right).
         \]
         \label{item:kesten2}
   \end{enumerate}
\end{lemma}
For our purposes, we will need the following variant of \pref{item:kesten1} adapted to our notion of backbone. 
 For $q \in \N,$ and $x,y \in \Z, z\in \Z^2$,
\begin{equation}\label{backboneestimate}
 \E\left[\big|\sfB_{G}[D_{x,y,z}^{(q)}]\big|\right]\lesssim q^2 \pi_2(q).
\end{equation}
It is a consequence of the following estimate (obtained by standard percolation-theoretic argument of encircling a box by an open circuit
to pass from the IIC to standard critical bond percolation and noting that the presence of an open circuit can only increase the size of the backbone:
\[
   \E\left[|\sfB_{G}[D_{x,y,z}^{(q)}]|\right]\lesssim q \sum_{i=1}^q \pi_2(i),
\]
along with the fact that $\sum_{i=1}^q \pi_2(i)\lesssim q \pi_2(q)$ (see \cite[Remark 37]{Kesten86iic}).

\begin{proof}[Proof of \pref{thm:iic-system}]
We recall that $\{X_n\}$ is the random walk on $G$ with $X_0 = \bm{0}$.

Consider now some square $D_{x,y,z}^{(q)} \in \overline{\cC}^{(k)}$ with $q=2^k$ and $k \geq 1$.
Set $\e \seteq 1/3$, and recalling the notion of depth from \eqref{eq:depth},
note that
\[
   \tau_G(D_{x,y,z}^{(q)}) \geq \e 2^k \implies \tau_G(D_{x,y,z}^{(q)}) \geq q/3 \implies \partial_G D_{3x+1,3y+1,z}^{(q/3)} \conn{G} \partial_G  D_{x,y}^{(q)}\,.
\]
Hence \pref{item:kesten2} implies that
\[
   \Pr\left[\cE(D,\lambda) \mid \overline{\cC}^{(k)} \right] \leq C \exp\left(-c \lambda^{1/8}\right),\qquad \forall D \in \overline{\cC}^{(k)},
\]
where
\[
   \cE(D,\lambda) \seteq \left\{\tau_G({D}) \geq \e 2^k \textrm{ and } |D\restrict_G\!\!| < \lambda^{-1} q^2 \pi_1(q)\right\}.
\]
In the sequel we will use $D_G$ to denote $D\restrict_G$.
It follows that conditioned on $\overline{\cC}^{(k)},$ for all $ D \in \overline{\cC}^{(k)}$
\begin{align*}
\E& \left[\frac{|\sfB_G[ D]|}{|D_G|} \1_{\{\tau_G({D}) \geq \e 2^k\}} \bigmid \overline{\cC}^{(k)}\right]  \\
     & \qquad \leq \frac{\lambda}{q^2 \pi_1(q)} (1-\Pr[\cE(D,\lambda) \mid \overline{\cC}^{(k)}]) \E\left[|\sfB_G[{D}]| \bigmid \lnot \cE(D,\lambda), \overline{\cC}^{(k)} \right] 
   + q^2 \Pr[\cE(D,\lambda)] \\
     & \qquad\leq \frac{\lambda}{q^2 \pi_1(q)} \E\left[|\sfB_G[{D}]|\right]
   + q^2 \Pr[\cE(D,\lambda)\mid \overline{\cC}^{(k)}] \\
   & \qquad\lesssim \lambda \frac{\pi_2(q)}{\pi_1(q)} 
   + q^2 \Pr[\cE(D,\lambda) \mid \overline{\cC}^{(k)}],
\end{align*}
where the second term in the first inequality uses the bound $|\sfB_G[{D}]\le q^2$, and in the last line we employed \eqref{backboneestimate}.   Choosing $\lambda \asymp (\log q)^{8}$ yields
\begin{align}
   \E\left[\frac{|\sfB_G[{D}]|}{|{D}_G|} \1_{\{\tau_G({D}) \geq \e 2^k\}}\bigmid \overline{\cC}^{(k)}\right]  
   \lesssim \frac{\pi_2(q)}{\pi_1(q)} (\log q)^8\,.\label{eq:ratio-bnd}
\end{align}

For any covering $\cC,$ and $D\in \cC,$ define the values
\begin{align*}
   f(G,x,y,D) &\seteq \1_{\{\tau_G(D) \geq \e 2^k\}} \1_D(x) \1_D(y) \frac{\1_{\sfB_G[D]}(y)}{|D_G|} \\
   F(G,x,y,\cC) &\seteq \sum_{D \in \cC} f(G,x,y,D).
\end{align*}
Let $\overline{\cC}^{(k)}_{\circ} \seteq \{ D \in \overline{\cC}^{(k)} : \tau_G(D) \geq \e 2^k \}$
denote the set of deep patches.
Recalling the definition from \eqref{backboneunion}, then we have:
\begin{align*}
   \E\left[\frac{1}{\deg_G(X_0)} \sum_{y \in V(G)} F(G,X_0,y,\overline{\cC}^{(k)})\right] &=
   \E\left[\sum_{D \in \overline{\cC}^{(k)}_{\circ}} \frac{|\sfB_G[ D]|}{|D_G|} \1_{D}(X_0)\right]
   \\
   \E\left[\frac{1}{\deg_G(X_0)} \sum_{y \in V(G)} F(G,y,X_0,\overline{\cC}^{(k)})\right] &=
   \E\left[\sum_{D \in \overline{\cC}^{(k)}_{\circ}} \1_{\sfB_G[ D]}(X_0) \1_{D}(X_0)\right] \geq \Pr\left[X_0 \in {\sfB}^{(k)}_{\circ}\right].
\end{align*}

Note that the last inequality is not an equality because of covering multiplicity---the event could be counted multiple times
in the middle expression.
Noting that $F(G,x,y,\overline{\cC}^{(k)})$ is an automorphism-invariant function of $(G,x,y,\psi_{G,\cC^{(k)}})$,
we may apply the mass transport principle \eqref{eq:mtp2}, yielding
\[
   \Pr\left[X_0 \in {\sfB}^{(k)}_{\circ}\right] \leq 
   \E\left[\sum_{D \in \overline{\cC}^{(k)}_{\circ}} \frac{|\sfB_G[D]|}{|D_G|} \1_{D}(X_0)\right]
   \stackrel{\eqref{eq:ratio-bnd}}{\lesssim} 
   \frac{\pi_2(q)}{\pi_1(q)} (\log q)^8\,,
\]
where in the last inequality we use the fact that $\mult(\overline\cC^{(k)}) \leq 10$ for every $k \geq 1$.
\end{proof}

To establish \pref{thm:speed-iic}, we need to
verify \pref{ass:system2}(4) for $d' > 2-\eta_1$.
To that end, we need the following result of Kesten.

\begin{theorem}[{\cite[Thm 8]{Kesten86iic}}]
   \label{thm:iic-density}
   For every $t \geq 1$, there is some constant $C_t \geq 1$ such that
   for any square $D_{x,y,z}^{(q)}$ containing the origin,
   \[
      \Pr\left[|D_{x,y,z}^{(q)} \cap V(G)| > \lambda q^2 \pi_1(q)\right] \leq C_t \lambda^{-t},\qquad \forall \lambda > 1.
   \]
\end{theorem}

\begin{proof}[Proof of \pref{thm:speed-iic}]
Fix some $\e > 0$ and take $d' \seteq 2 - \eta_1 + \e$.
For $k \geq 1$, denote $q \seteq 2^k$ and $\lambda \seteq 2^{\e k}$.
Then \pref{thm:iic-density} gives
\[
   \Pr[\bm{0} \in \cL^{(k)}] \lesssim C_t 2^{-\e k t}.
\]
Setting $t \seteq \eta/\e$, where $\eta = \eta_{21}$, gives a bound of the form \eqref{eq:big-patch}.
Now applying \pref{thm:thick-thin2} yields the desired result.
\end{proof}

\subsection*{Acknowledgements}
The authors thank two anonymous referees for several helpful comments. 
They are also grateful to the Simons Institute, where
this research was initiated. SG was partially supported by NSF DMS-1855688, NSF CAREER Award DMS-1945172, and a Sloan Research Fellowship.
JL was partially supported by NSF CCF-1616297 and a Simons Investigator Award.

\bibliographystyle{alpha}
\bibliography{diffusive}

\end{document}